\documentclass[preprint,11pt]{elsarticle}
\usepackage{amsmath,amssymb,calrsfs,mathrsfs,amsthm}

\usepackage{amsrefs}
\usepackage{mathtools}
\usepackage{enumerate}
\usepackage{aliascnt}
\usepackage{latexsym}
\usepackage{extarrows}
\usepackage[symbol]{footmisc}
\usepackage{multirow}
\usepackage{graphicx}
\usepackage[colorlinks,linkcolor=blue]{hyperref}

\numberwithin{thmm}{section}

\def\NewTheorem#1#2{%
	\newaliascnt{#1}{thmm}
	\newtheorem{#1}[#1]{#2}
	\aliascntresetthe{#1}
	\expandafter\def\csname #1autorefname\endcsname{#2}
}
\numberwithin{equation}{section}

\makeatletter
\def\namedlabel#1#2{\begingroup
   \def\@currentlabel{#2}%
   \label{#1}\endgroup
}
\makeatother

\NewTheorem{lem}{Lemma}
\NewTheorem{thm}{Theorem}
\NewTheorem{cor}{Corollary}
\NewTheorem{prop}{Proposition}
\theoremstyle{definition}
\NewTheorem{defi}{Definition}
\theoremstyle{remark}
\NewTheorem{rem}{Remark}

\makeatletter
\def\namedlabel#1#2{\begingroup
   \def\@currentlabel{#2}%
   \label{#1}\endgroup
}
\makeatother
\begin{document}
\begin{frontmatter}
\title{Global regularity criteria for the Navier-Stokes equations based on one approximate solution}

\author{Tuan N.\ Pham}
\address{Department of Mathematics, Brigham Young University,\\275 TMCB, Provo UT 84606, USA\\tuan.pham@mathematics.byu.edu}
\date{\today}


\begin{abstract} Considering the three-dimensional incompressible Navier-Stokes equations on the whole space, we address the question: is it possible to infer the global regularity of a mild solution from a single approximate solution? Assuming a relatively simple scale-invariant relation involving the size of the approximate solution, the resolution parameter, and the initial energy, we show that the answer is affirmative for a general class of approximate solutions, including the Leray's mollified solutions. Two different treatments leading to essentially the same conclusion are presented.\end{abstract}
\end{frontmatter}

\section{Introduction}
We consider the Cauchy problem for the three-dimensional incompressible Navier-Stokes equations:
\[{\rm(NSE)}:~\left\{ {\begin{array}{*{20}{rcl}}
   {{\partial _t}u - \Delta u + u \nabla u + \nabla p = 0}&~~{\rm in}&\mathbb{R}^3\times(0,\infty),  \\
   {{\rm div}\, u = 0}&~~{\rm in}&\mathbb{R}^3\times(0,\infty),  \\
   {u(\cdot,0) = {u_0}}&~~{\rm in}&\mathbb{R}^3.  \\
\end{array}} \right.\]
While the global-wellposedness of (NSE) is still not known, a variety of regularized systems obtained by mollifying the nonlinear term are known to be globally-wellposed (see e.g. \cite[Sec.\ V]{leray}, \cite[Sec.\ 6.3]{lady69}, \cite[Ch.\ III, \S 1]{temam}). Regularizations of the nonlinear term often involve a resolution parameter $\varepsilon$. Two well-known examples are:
\begin{itemize}
\item[1.] The classical Leray regularization $(u*\eta_{\varepsilon}) \nabla u$,\vspace{-.4em}
\item[2.] The regularization $P_\varepsilon (u\nabla u)$, where $P_\varepsilon$ is an orthogonal projection on $L^2(\mathbb{R}^3)$ whose Fourier multiplier is a smooth cutoff function supported on the ball $B_{2\varepsilon^{-1}}$ and equal to 1 on the ball $B_{\varepsilon^{-1}}$.\vspace{-.4em}
\end{itemize}
These approximations generate a family of global smooth approximate solutions to (NSE), which can be useful for the construction of global weak solutions. Full information about the behavior of a sequence of approximate solutions as $\varepsilon\downarrow 0$ would give useful information about the exact solution. However, in practice we only have information about finitely many approximate solutions. Let us in fact assume that we know \emph{only one} approximate solution for a certain value of $\varepsilon$. 
\begin{quotation}\vspace{-.4em}\noindent
Under what conditions can we infer global regularity for the exact solution? 
\end{quotation}
\vspace{-.4em}
The question was studied in \cite{buyang, titi2007, dashti, morosi, livio} in various settings. For example, Li \cite{buyang} considered a discretized Navier-Stokes system on a polyhedron and showed that: if the numerical solution $u_\varepsilon$ corresponding to some mesh size $\varepsilon$ is of size $M$ (the $L^\infty$-norm of $u_\varepsilon$) and $\varepsilon\lesssim\exp(-(\|u_0\|_{H^1_0\cap H^2}+1)^\alpha M^{\alpha})$, where $\alpha$ is a large number, then the exact solution is regular for all times. A global regularity criterion of this kind is generally referred as \emph{posteriori regularity criterion}, which serves as a check for an approximate/numerical solution to guarantee the existence of the exact strong solution. It differs from other well-known global regularity criteria which often require either a smallness condition on the initial data (e.g.\ \cite{kato1984, tataru}) or a subcritical/critical control on the exact solution (e.g.\ Prodi-Serrin-Ladyzhenskaya criteria, Beale-Kato-Majda criterion, \cite{sverak2003, kukavica2005, evan, titi2011}), or some special structure of the initial data (e.g.\ \cite{lady68, seregin2020, kukavica2013, shi18, jiasverak2014, chemin11}). 

In this paper, we investigate criteria for global regularity based on one continuous approximate solution on the whole space $\mathbb{R}^3$. In this setting, the problem already contains the key difficulties but is technically simpler. We consider the following regularized Navier-Stokes system:
\[\left\{ {\begin{array}{*{20}{rcl}}
   {{\partial _t}u - \Delta u + [u\nabla u]_\varepsilon + \nabla p = 0},  \\
   {{\rm div}\, u = 0},  \\
   {u(\cdot,0) = {u_0}}  \\
\end{array}} \right.\]
where $[u\nabla u]_{\varepsilon}$ is used as a common notation for two types of regularizations mentioned above. We give a simple criterion involving  the resolution parameter $\varepsilon$, the size $M$ (the $L^\infty$-norm) of the corresponding approximate solution, and the total energy $\|u_0\|^2_{L^2}$ which guarantees that the exact solution is regular globally. The overall idea is to combine the energy method, which treats temporal asymptotic behaviors, with the perturbation method, which treats finite-time behaviors. The heuristics is quite simple: first, it is well-known that every Leray-Hopf weak solution corresponding to the same initial data becomes smooth and decaying after some time $T_0=C\|u_0\|_{L^2}^4$, where $C$ is an absolute constant. Secondly, if $\varepsilon$ is sufficiently small depending on $T_0$ and $M$ then $u_\varepsilon$ stays close to the exact solution $u$ up to time $T_0$. By these arguments, it is quite conceivable that a strong solution should exist globally and be of size $2M$.

Another natural question is how large $\varepsilon$ can be in terms of the ``observable'' quantity $M$ to still guarantee global regularity of exact solutions. The scaling symmetry turns out to be useful to predict possible answers. Recall the scaling symmetry:
\[u(x,t)\to \lambda u(\lambda x,\lambda^2 t),~~~p(x,t)\to \lambda^2 p(\lambda x,\lambda^2 t),~~~u_0(x)\to \lambda u_0(\lambda x).\]
Often, the resolution parameter $\varepsilon$ can be normalized to have the same scaling as spatial length. Since both $\varepsilon$ and $M^{-1}$ are length scales of the problem, it seems reasonable to pursue the rate of $\varepsilon\lesssim M^{-1}$. However, the time-dependence nature complicates the problem. For one reason, the initial energy $\|u_0\|_{L^2}^2$ also has the same scaling as spatial length and, thus, can be considered as another length scale of the problem. Another reason is that the higher initial energy naturally requires finer resolution in order to capture complex structures of the exact solution at small scales. Our main purpose is to derive a scale-invariant condition essentially of the form $\varepsilon\lesssim M^{-1}\exp(-\|u_0\|_{L^2}^4M^2)$. This is an improvement of the condition in \cite{buyang}, although one should haste to add that our regularizations are different and, in particular, infinite-dimensional. Our criterion is also different from those in \cite{buyang, titi2007, dashti, morosi, livio} in that it is scale-invariant and does not depend on the size of the approximate solution's derivatives.

We present two different methods, one at a global scale in space and time, and one at a local scale, both leading to essentially the same result. In fact, our methods are sufficient to deal with more general approximate Navier-Stokes systems of the form:
\[{\rm(NSE)_\varepsilon}:~\left\{ {\begin{array}{*{20}{rcl}}
   {{\partial _t}u - \Delta u + u \nabla u + \nabla p = f_\varepsilon+ {\rm div}\,g_\varepsilon},  \\
   {{\rm div}\, u = 0},  \\
   {u(\cdot,0) = {u_{0\varepsilon}}}.  \\
\end{array}} \right.\]
The regularized Navier-Stokes systems mentioned above are special cases of (NSE)$_\varepsilon$. Indeed, in the first case, $g_\varepsilon=(u_\varepsilon-u_\varepsilon*\eta_\varepsilon)\otimes u_\varepsilon$, where $u_\varepsilon$ is the solution of the corresponding regularized system. In the second case, $g_\varepsilon=(Id-P_\varepsilon)(u_\varepsilon\otimes u_\varepsilon)$. In both cases, $u_{0\varepsilon}=u_0$ and $f_\varepsilon=0$. In the global approach, we use the smallness of the error terms $f_\varepsilon$ and $g_\varepsilon$ in a global sense to derive a simple proof based on perturbation method. The local approach is an application of the $\epsilon$-regularity criterion, a type of criteria for the local smoothness of a weak solution introduced by Caffarelli, Kohn and Nirenberg \cite{caff}. 
Besides serving as an alternate methodology to analyze the problem, the local approach allows the error terms to be small in a local sense. 
The main difficulty is to control the error between the approximate solution and the exact solution as it grows over time. As to be shown, this error grows at most exponentially in time.
Our main results are the following:
\begin{thm}[Global scale]\label{mainthm1}
Let $u_0\in L^\infty$ be a divergence-free initial data. For some $\varepsilon$, $M$, $T>0$, let $f_\varepsilon$, $g_\varepsilon$ be functions such that for every $\sigma\ge0$
\begin{equation}\label{globalc}\left\{ \begin{array}{*{35}{l}}
   Ff^\sigma_\varepsilon,~Gg^\sigma_\varepsilon\in C([0,T ),{{L}^{\infty }}),  \\
   \lim_{t\to 0}\,{{\left\| Ff^\sigma_\varepsilon(t) \right\|}_{{{L}^{\infty }}}}=\lim_{t\to 0}\,{{\left\| Gg^\sigma_\varepsilon(t) \right\|}_{{{L}^{\infty }}}}=0,  \\
   {{\left\| Ff^\sigma_\varepsilon \right\|}_{L^{\infty }}},~{{\left\| Gg^\sigma_\varepsilon \right\|}_{L^{\infty }}},~\|u_{0\varepsilon}-u_0\|_{L^\infty}\lesssim\varepsilon {{M}^{2}}.  \\
\end{array} \right.\end{equation}
Suppose (NSE)$_\varepsilon$ has a mild solution $u_\varepsilon$ satisfying $\|u_\varepsilon\|_{L^\infty(\mathbb{R}^3\times(0,T))}\le M$. Then there exist absolute constants $\delta_1$, $\mu_1>0$ such that if $\varepsilon\le \delta_1M^{-1}\exp(-\mu_1TM^2)$ then (NSE) has a mild solution $u$ on $(0,T)$ with $\|u\|_{L^\infty(\mathbb{R}^3\times(0,T))}\le 2M$.
\end{thm}
\begin{thm}[Local scale]\label{mainthm2}
Let $q_1$, $q_2$, $\nu_i$, $\sigma_i$, $\lambda_i$, for $i=1,2,3$, be such that
\begin{equation}\label{condensedcond}
\left\{ \begin{matrix}
   5/2<{{q}_{1}}\le \infty ,\ 5<{{q}_{2}}\le \infty ,\\
   {{\nu }_{1}}+{{\nu }_{2}}-{{\nu }_{3}} = -3+5/{{q}_{1}},  \\
   {{\sigma }_{1}}+{{\sigma }_{2}}-{{\sigma }_{3}} = -2+5/{{q}_{2}} , \\
      {{\lambda }_{1}}+{{\lambda }_{2}}-{{\lambda }_{3}} = 1/2 , \\
   {{\nu }_{1}},\,{{\sigma }_{1}},\,{{\lambda }_{1}}>0,\,{{\nu }_{i}},\,{{\sigma }_{i}},\,{{\lambda }_{i}}\ge 0.   \\
\end{matrix} \right.\end{equation}
Let $u_0\in L^\infty$ be a divergence-free initial data. For some $\varepsilon,\,M>0$, let $f_\varepsilon$ and $g_\varepsilon$ be functions satisfying
\begin{eqnarray}\label{localc1}
\|f_\varepsilon\|_{L^{q_1}(B_r(x)\times(t-r^2,t))}&\lesssim&\varepsilon^{\nu_1}r^{\nu_2}M^{\nu_3},\\
\label{localc2}\|g_\varepsilon\|_{L^{q_2}(B_r(x)\times(t-r^2,t))}&\lesssim&\varepsilon^{\sigma_1}r^{\sigma_2}M^{\sigma_3},\\
\label{localc3}\|u_{0\varepsilon}-u_0\|_{L^{2}(B_r(x))}&\lesssim&\varepsilon^{\lambda_1}r^{\lambda_2}M^{\lambda_3}
\end{eqnarray}
for all $x\in\mathbb{R}^3$, $r>0$, $t>r^2$. Suppose (NSE)$_\varepsilon$ has a weak solution $u_\varepsilon$ satisfying 
\begin{gather}\label{uecond}
\nonumber u_\varepsilon \in C([0,T),L^\infty),\ \ \ \nabla u_\varepsilon \in L^2((0,T),L^2_{\text{uloc}}(\mathbb{R}^3)),\\
\|u_\varepsilon\|_{L^\infty(\mathbb{R}^3\times(0,T))}\le M.
\end{gather} 
Then there exist constants $\delta_2$, $\mu_2>0$ depending on $\nu_1,\sigma_1,\lambda_1,q_1,q_2$ such that if $\varepsilon\le \delta_2M^{-1}\exp(-\mu_2TM^2)$ then (NSE) has a mild solution $u$ on $(0,T)$ with $\|u\|_{L^\infty(\mathbb{R}^3\times(0,T))}\le C_{q_1,q_2}M$.
\end{thm}
For the two classical regularized systems mentioned above, the condition \eqref{globalc} is satisfied. It should be noticed that if $\|u_0\|_{L^\infty}\le M$ then the exact solution remains bounded up to some time $\gtrsim M^{-2}$ without any further condition. Starting from this time, conditions \eqref{condensedcond}-\eqref{localc3} are satisfied. See \autoref{mol1} and \autoref{mol2} for detail justification. We remark that all quantitative relations in the above theorems are scaling invariant.

In the case $u_0\in L^2$, every Leray's weak solution eventually becomes regular after $T_0$ and decays as $\|u(t)\|_{L^\infty}\lesssim \|u_0\|_{L^2}t^{-3/4}$. This result is due to Leray \cite[Para.\ 34]{leray}. Therefore, if the mild solution to (NSE) blows up in finite time, it must blow up before this time. One obtains a consequence of the main theorems as follows.

\begin{cor}\label{globalsol}
Let $u_0\in L^2\cap L^\infty$. Suppose that the hypotheses of \autoref{mainthm1} (or \autoref{mainthm2}) hold for $T=T_0$. 
There exist absolute constants $C_1$, $C_2>0$ such that if $\|u_\varepsilon\|_{L^\infty(\mathbb{R}^3\times(0,T_0))}\le M$ for some \[\varepsilon\le C_1M^{-1}\exp\left(-C_2 \|u_0\|_{L^2}^4M^2\right)\] then $u$ is a global solution and $\|u\|_{L^\infty(\mathbb{R}^3\times(0,\infty))}\le \tilde{C}M$, where the constant $\tilde{C}=2$ (or $\max\{2,\,C_{q_1,q_2}$\}, respectively).
\end{cor}

\noindent\textbf{Notations.} We will denote by $C$ a positive absolute constant, by $C_{\alpha,\lambda,...}$ or $C(\alpha,\gamma,...)$ a positive constant depending on $\alpha$, $\lambda$, and so on. Nonessential constants $C$ may vary from line to line. We write $X\lesssim Y$ if $X\le CY$ for some positive absolute constant $C$. In addition, the following notations are used:
\begin{itemize}
\item[] $B_r(x)=\{y\in\mathbb{R}^3:\,|y-x|<r\}$,\ \ \ $B_r=B_r(0),$
\item[] $Q_{r,\theta}(z)=B_r(x)\times(t-\theta r^2,t)$, ~~$z=(x,t)$,
\item[] $Q_r(z)=B_r(x)\times(t-r^2,t)$,~~ $z=(x,t)$,
\item[] $Q_r=Q_r(0,0)$,\ \ \ $Q_{r,\theta}=Q_{r,\theta}(0,0)$,
\item[] $u\otimes v = (u_iv_j)$,\ \ \ $a:b=a_{ij}b_{ij}$,\ \ \ div\,$g=(g_{ij,j})$ with Einstein summation convention,
\item[] $f^\sigma(x,t)=f(x,t+\sigma)$,\ \ \ $[f]_B=\frac{1}{|B|}\int_Bfdx$,
\item[] $\mathbb{P}$ denotes Leray projection onto divergence-free fields on $\mathbb{R}^3$,
\item[] $Ff=\int_0^t e^{(t-s)\Delta}\mathbb{P}f(s)ds$,\ \ \  $Gg=\int_0^t e^{(t-s)\Delta}\mathbb{P}{\rm div\,}g(s)ds$,\ \ \ $B(u,v)=-G(u\otimes v)$,
\item[] $\|u\|_{C^\alpha_{\text{par}}(Q)}=\|u\|_{L^\infty(Q)}+[u]_{C^\alpha_{\text{par}}(Q)}$,\ \ \ \ $[u]_{C^\alpha_{\text{par}}(Q)}=\underset{(x',t'),(x,t)\in Q}{\mathop{\sup }}\,\frac{|u(x',t')-u(x,t)|}{{{\left( |x'-x{{|}^{2}}+|t'-t| \right)}^{\alpha /2}}}$.
\end{itemize}

\section{At a global scale}
Our goal in this section is to give a proof of \autoref{mainthm2}. We recall a useful abstract lemma for Banach spaces:
\begin{lem}\label{abstractlem}
	Let $\mathcal{X}$ be a Banach space, $L:\mathcal{X}\to\mathcal{X}$ be a continuous linear map with $\|L\|\le\lambda<1$, and $B:\mathcal{X}\times \mathcal{X}\to \mathcal{X}$ be a continuous bilinear map with $\|B\|\le\gamma$. For $a\in \mathcal{X}$, consider the fixed point problem
$x=a+Lx+B(x,x)$. If $4\lambda\|a\|<(1-\lambda)^2$ and $0<r_1<r_2$ are two roots of the equation $r=\|a\|+\lambda r+\gamma r^2$, then the problem has a unique solution $\bar{x}$ in the ball $\{x:~\|x\|< r_2\}$. Moreover, $\|\bar{x}\|\le r_1$. It is given as the limit of the sequence $(x_n)$ with $x_0=0$, $x_{n+1}=a+B(x_n,x_n)$.
\end{lem}
We refer the readers to \cite[Lemma A.1]{gallagher2003} for the proof. A mild solution of the Navier-Stokes equation with right hand side $f+{\rm div\,} g$ is defined by Picard's iteration from the equation
\begin{equation}\label{37191}u=e^{t\Delta}u_0+Ff+Gg+B(u,u).\end{equation}
Suppose $u_0\in L^\infty$ and $Ff$, $Gg\in C([0,\infty),L^\infty)$. The mild solution, if exists on an interval $(0,T)$, belongs to the class $C([0,T),L^\infty)$; see e.g.\ \cite[Sec.\ 3]{pooley}. \autoref{abstractlem} together with basic estimates on the convolution kernel of the Stokes operator gives the following result due to Leray \cite[Para.\ 19]{leray}: 
\begin{rem}\label{constantc}
For $u_0\in L^\infty(\mathbb{R}^3)$, $f=0$ and $g=0$, there exists an absolute constant $C_0>0$ such that for $T=C_0\|u_0\|_{L^\infty}^{-2}$ the equation \eqref{37191} has a solution $u\in L^\infty(\mathbb{R}^3\times(0,T))$ and $\|u\|_{L^\infty}\le 2\|u_0\|_{L^\infty}$.
\end{rem}

\subsection{Proof of \autoref{mainthm1}}
\noindent We will estimate the growth in $L^\infty$-norm of $v=u-u_\varepsilon$ after each time step of size $\tau\lesssim M^{-2}$. As to be shown, $v$ is roughly speaking increased by at most 4 times after each time step, resulting in the hypothetical exponential relation between $\epsilon$ and $M$. 
\begin{lem}\label{prop28191}
Let $I=[0,\tau]$. Consider a continuous nondecreasing function $\varphi:I\to\mathbb{R}$ satisfying
\[\varphi(t)\le\varphi(0)+\alpha+\beta\varphi(t)+\gamma\varphi(t)^2~~~~\forall~ t\in I,\]
for some constants $\alpha$, $\beta$, $\gamma>0$. Suppose $\beta<\frac{1}{2}$ and $\varphi(0)+\alpha<\frac{1}{16\gamma}$. Then $\varphi(\tau)<4(\varphi(0)+\alpha)$.
\end{lem}
\begin{proof}
Put $\lambda=\varphi(0)+\alpha$. Suppose by contradiction that $\varphi(\tau)\ge 4\lambda> \varphi(0)$. By the continuity of $\varphi$, there exists $t\in I$ such that $\varphi(t)=4\lambda$. Then
\[4\lambda =\varphi (t)\le \varphi (0)+\alpha +\beta \varphi (t)+\gamma \varphi {{(t)}^{2}}<\lambda +\frac{1}{2}(4\lambda )+\gamma {{(4\lambda )}^{2}},\]
which implies $1<16\gamma\lambda$. This is a contradiction.
\end{proof}
\begin{proof}[Proof \autoref{mainthm1}]
Let $C_0>0$ be the absolute constant identified in \autoref{constantc}.
Put $\tau=\frac{\theta^2}{M^2}$, where $\theta<\min\{1,\sqrt{C_0}/2\}$ is a small absolute constant to be chosen later. 
For $\sigma\ge 0$, denote by $u^\sigma$, $u_\varepsilon^\sigma$, $f^\sigma_\varepsilon$, $g^\sigma_\varepsilon$ the time-shifted version of  $u$, $u_\varepsilon$, $f_\varepsilon$, $g_\varepsilon$, respectively, i.e. $u^\sigma(x,t)=u(x,t+\sigma)$,\ldots
If $u$ is bounded on $[\sigma,\sigma+\tau]$ for some $\sigma>0$ then the difference $v^\sigma=u^\sigma-u^\sigma_\varepsilon$
solves
\begin{equation}
\label{28193}v^\sigma(t)=e^{t\Delta}v^\sigma(0)-F f^\sigma_\varepsilon-G g^\sigma_\varepsilon+B(v^\sigma,{{u}^\sigma_{\varepsilon }})+B({{u}^\sigma_{\varepsilon }},v^\sigma)+B(v^\sigma,v^\sigma).\end{equation}
Taking the $L^\infty_x$-norm of \eqref{28193} leads to
\begin{equation}\label{28194}{{\left\| v^\sigma(t) \right\|}_{L^{\infty }}}\le {{\left\| v^\sigma({0}) \right\|}_{L^{\infty }}}+C\varepsilon {{M}^{2}}+CM\int_{{0}}^{t}{\frac{{{\left\| v^\sigma(s) \right\|}_{L^{\infty }}}}{\sqrt{t-s}}ds}+C\int_{{0}}^{t}{\frac{\left\| v^\sigma(s) \right\|_{L^{\infty }}^{2}}{\sqrt{t-s}}ds}.\end{equation}
Because the function $\varphi_\sigma(s)=\|v^\sigma\|_{L^\infty(\mathbb{R}^3\times[0,s])}$ is a continuous and nondecreasing on $[0,\tau]$, \eqref{28194} implies
\begin{equation}\label{28195}\varphi_\sigma(t)\le\varphi_\sigma(0)+\alpha+\beta\varphi_\sigma(t)+\gamma\varphi_\sigma(t)^2~~~~\forall~ t\in [0,\tau],\end{equation}
where $\alpha=C\varepsilon M^2$, $\beta=C\theta$, $\gamma=\frac{C\theta}{M}$. We choose $\theta$ such that $\beta=C\theta<\frac{1}{2}$ and that $K=\frac{1}{\theta^2}TM^2 +1$ is an integer. Assume the relation $\varepsilon\le \delta_1M^{-1}\exp(-\mu_1TM^2)$ where $\delta_1$ and $\mu_1$ are to be determined. Take $\mu_1=\frac{4}{\theta^2}$ and choose $\delta_1$ sufficiently small such that ${{4}^{2K}}\alpha ={{4}^{2K}}(C\varepsilon {{M}^{2}})<M$ and $\alpha \gamma =C\varepsilon M<{{4}^{-2K-3}}$.
For $k\ge 1$, denote $t_k=(k-1)\tau$. Note that $t_K=T$. We show by induction on $2\le j\le K$ that
\begin{enumerate}[~~~~~~~(a)]
\item 
$\|u\|_{L^\infty({\mathbb{R}^3\times[0,t_j]})}\le 2M$,
\item $\varphi_0(t_j)<4^{2j}\alpha$.
\end{enumerate}
Note that $\varphi_0(0)= \|u_{0\varepsilon}-u_0\|_{L^\infty}\lesssim \varepsilon M^2$. By scaling $\alpha$ by a constant factor if necessary, one can assume $\varphi_0(0)\le\alpha$. Because $\beta<\frac{1}{2}$ an $\alpha\gamma<\frac{1}{16}$, one can apply \autoref{prop28191} for $\varphi=\varphi_0$. Then $\varphi_0(\tau)<5\alpha$, and thus (a) and (b) are true for $j=2$. Suppose that (a) and (b) hold for $2\le j=k<K$. By our choice of $\tau$, $u$ stays bounded on the interval $[t_k,t_k+\tau]=[t_k,t_{k+1}]$. 
Moreover,
\[{{\varphi }_{{{t}_{k}}}}(0)+\alpha \le {{\varphi }_{0}}({{t}_{k}})+\alpha <{{4}^{2k}}\alpha +\alpha <{{4}^{2k+1}}\alpha <\frac{1}{16\gamma }.\]
Applying \eqref{28195} for $\sigma=t_k$ and \autoref{prop28191} for $\varphi=\varphi_{t_k}$, we conclude that ${{\varphi }_{{{t}_{k}}}}(\tau)\le 4({{\varphi }_{{{t}_{k}}}}(0)+\alpha )<{{4}^{2k+2}}\alpha$. Thus, (b) is true for $j=k+1$. Moreover,
\[{{\left\| u({{t}_{k+1}}) \right\|}_{{{L}^{\infty }}}}\le {{\left\| {{u}_{\varepsilon }}({{t}_{k+1}}) \right\|}_{{{L}^{\infty }}}}+{{\left\| v({{t}_{k+1}}) \right\|}_{{{L}^{\infty }}}}\le M+{{\varphi }_{0}}({{t}_{k+1}})\le M+4^K\alpha< 2M.\]
Therefore, (a) is true for $j=k+1$. 
\end{proof}


The following arguments are sufficient to prove \autoref{globalsol}.

By \cite[Para.\ 21]{leray}, if the dimensionless quantity $\|u_0\|_{L^\infty}^2\|u_0\|_{L^2}$ is sufficient small then $\|u(t)\|_{L^\infty}\le 2\|u_0\|_{L^\infty}$ for all $t>0$. On the other hand, $\|u(t)\|_{L^\infty}\lesssim \|u_0\|_{L^2}t^{-3/4}$ for all $t\ge T_0=C\|u_0\|_{L^2}^4$. One can increase $C$ if necessary such that $\|u(t)\|_{L^\infty}^2\|u_0\|_{L^2}$ is small for all $t\ge T_0$. Then $\|u(T_0)\|_{L^\infty}^2\|u(T_0)\|_{L^2}$ is small thanks to the energy inequality $\|u(t)\|_{L^2}\le \|u_0\|_{L^2}$. Therefore, $\|u(t)\|_{L^\infty}\le 2\|u(T_0)\|_{L^\infty}\le 2M$ for all $t\ge T_0$. 


\section{At a local scale}
Our goal in this section is to give a proof of \autoref{mainthm2}. We illustrate our method heuristically as follows:

If $\|u(t)\|_{L^\infty}\le{2\tilde{C}M}$ for some $t$ then $u$ is regular at least until some time $t+\delta$. On the time interval $I=(t,t+\delta)$, the function $v=u-u_\epsilon$ belongs to a class of weak solutions to a generalized Navier--Stokes system. We use an $\epsilon$-regularity criterion for $v$ to show that $v$ is bounded by $\tilde{C}M$ on $I$. Then $\|u\|_{L^\infty(\mathbb{R}^3\times I)}\le \|u_\varepsilon\|_{L^\infty}+\|v\|_{L^\infty(\mathbb{R}^3\times I)}\le (1+\tilde{C})M\le 2\tilde{C}M$. We then continue this argument with $t$ being replaced by $t+\delta$.

We will formulate an $\epsilon$-regularity criterion for suitable weak solutions (to be defined) for the generalized Navier--Stokes equations. Key to the proof of \autoref{mainthm2} is to control the growth of the local energy of $v$ after each time step of size $\lesssim M^{-2}$. Roughly speaking, this energy is increased by at most 4 times after each time step. 

\subsection{Suitable weak solutions to a generalized NSE}
We define suitable weak solutions to the generalized Navier-Stokes system
\[{\rm(GNSE):}\left\{ \begin{matrix}
{{\partial }_{t}}u-\Delta u+\operatorname{div}(a\otimes u+u\otimes a+u\,\otimes u)+\nabla p=f+\operatorname{div}g,  \\
\operatorname{div}u=0  \\
\end{matrix} \right.\ \ \ (x,t)\in \mathcal{O}\]
as follows.
\begin{defi}\label{generalizedsol}
Let $\mathcal{O}=D\times I=D\times(c,d)$ be an open subset of $\mathbb{R}^3\times\mathbb{R}$. Let $f\in L^{q_1}(\mathcal{O})$, $g\in L^{q_2}(\mathcal{O})$, $a\in L^m(\mathcal{O})$, div\,$a=0$ with $q_1>5/2$, $q_2>5$, $m>5$. A pair of functions $(u,p)$ is said to be suitable weak solution to (GNSE) if the following conditions are satisfied.
\begin{enumerate}[(i)]
	\item $u\in L^\infty_t L^2_x\cap L^2_t\dot{H}^1_x(\mathcal{O}')$ and $p\in L^{3/2}(\mathcal{O}')$ for any bounded set $\mathcal{O}'\subset\mathcal{O}$. Moreover, 
	\[\underset{{{z}}\in \mathcal{O}}{\mathop{\sup }}\,\left( {{\left\| u \right\|}_{L_{t}^{\infty }L_{x}^{2}(\mathcal{O}\cap {{Q}_{1}}({{z}}))}}+{{\left\| \nabla u \right\|}_{{{L}^{2}}(\mathcal{O}\cap {{Q}_{1}}({{z}}))}} \right)<\infty. \]
	\item They satisfy the equations $\partial_tu-\Delta u+{\rm div}(a\otimes u+u\otimes a+u\otimes u)+\nabla p=f+{\rm div}\,g$ and div\,$u=0$ in sense of distribution on $\mathcal{O}$; that is, for each $\psi\in C^\infty_0(\mathcal{O})$
\begin{eqnarray*}\int_{{D}}{u(x,t)\psi (x,t)dx}=\int_{c }^{t}{\int_{{D}}{\left[ u\left( {{\partial }_{s}}\psi +\Delta \psi  \right)+(a\otimes u+u\otimes a+u\otimes u)\nabla \psi   \right.}}\\
\left. +p\operatorname{div}\psi+f\psi -g\nabla \psi  \right]dxds,
\end{eqnarray*}
\[\int_{{{\mathbb{R}}^{3}}}{u(x,t)\cdot \nabla \psi (x,t)dx}=0~~~~\forall\,t\in I.\]
	\item The following local energy inequality holds for all nonnegative $\psi\in C^\infty_0(\mathcal{O})$. 
	\begin{eqnarray}
\nonumber	\int_{\mathcal{O}}{|\nabla u{{|}^{2}}\psi dxds}\le \int_{\mathcal{O}}{\left[ \frac{|u{{|}^{2}}}{2}({{\partial }_{t}}\psi +\Delta \psi )+\left( \frac{|u{{|}^{2}}}{2}+p \right)u\nabla \psi dxds \right.}\\
\nonumber	+\frac{|u{{|}^{2}}}{2}a\nabla \psi +u\otimes a:\psi \nabla u+u\otimes u:a\nabla \psi \\
\label{323191}	\left. +uf\psi +g:\psi \nabla u+g:u\otimes \nabla \psi  \right]dxds
\end{eqnarray}
	
\end{enumerate}
\end{defi}
\begin{rem}
From Part (ii) of \autoref{generalizedsol}, it can be seen that for any $\psi\in C^\infty_0(D\times I)$ the map $t\mapsto \int_{{D}}{u(x,t)\psi (x,t)dx}$ is continuous on $I$. By choosing $\psi(x,t)=\chi(t)\phi(x)$ where $\chi\in C^\infty_0(I)$ and $\phi\in C^\infty_0(D)$, one concludes that the map $t\mapsto \int_{{{\mathbb{R}}^{3}}}{u(x,t)\phi (x)dx}$ is continuous on $I$. Denote by $(L^2(D),w)$ the space $L^2(D)$ equipped with weak topology. Together with the fact that $u\in L^\infty_tL^2_x(D\times I)$, the map $u:I\to (L^2(D),w)$ is continuous. This observation leads to the following property.
\end{rem}

\begin{lem}\label{prop131192}
Let $u$ be a suitable weak solution on $D\times I=D\times (c,d)$. Let $\psi\in C(D\times\mathbb{R})$, $\psi\ge 0$, compactly supported on $D\times \mathbb{R}$. Then the map $\xi:I\to\mathbb{R}$
\[\xi(t)=\int_{D}{\frac{|u(x,t){{|}^{2}}}{2}\psi (x,t)dx}\]
is lower semi-continuous, i.e. $\underset{s\to t}{\mathop{\lim \inf }}\,\xi(s)\ge \xi(t)$ for every $t\in I$.
\end{lem}

Although we only assume a local $L^2$-bound on the gradient of the solution in the definition of suitable weak solutions, a similar bound on the solution itself can be inferred from \noindent\autoref{prop131192}:
\begin{equation}\label{326191}
\int_D\frac{|u(x,t)|^2}{2}\psi(x,t)dx+\int_c^t\int_D|\nabla u|^2\psi dxds\le \textrm{RHS}\eqref{323191}
\end{equation}
for every function $0\le\psi\in C^\infty_0(D\times I)$ and $t\in I=(c,d)$. A simple proof of this fact can be obtained by standard cut-off argument (see also \cite[p.\ 788]{caff}): for fixed $t\in I$, apply the test functions $\psi_\delta(x,s)=\psi(x,s)\chi\left(\frac{t-s}{\delta}\right)$ where $\chi:\mathbb{R}\to\mathbb{R}$ is a smooth nondecreasing function such that $\chi(s)=0$ for $s\le 0$, and $\chi(s)=1$ for $s\ge 1$. Then let $\delta\to 0$.

\begin{lem}\label{prop21191}
Let $u$ be a suitable weak solution on $\mathbb{R}^3\times I$ where $I=(c,\,d)$. Let $\phi\in C^\infty_0(\mathbb{R}^3)$, $\phi\ge 0$. We have the following statements.
\begin{enumerate}[(i)]
\item $\xi(t)+\rho(t)\le \xi(t_0^+)+\rho(t_0)+k(t_0,t)$ for any $c\le t_0<t<d$, where
\begin{eqnarray*}
\xi(t)&=&\int_{\mathbb{R}^3}{\frac{|u(x,t){{|}^{2}}}{2}\phi (x)dx},\\
\rho(t)&=&\int_c^t \int_D|\nabla u|^2\phi dxds,\\
k(t_1,t_2)&=& \int_{t_1}^{t_2}{\int_{\mathbb{R}^3}{\left[ \frac{|u{{|}^{2}}}{2}\Delta \phi+\left( \frac{|u{{|}^{2}}}{2}+p \right)u\nabla \phi+ \right.}}\\
\nonumber&&+\frac{|u{{|}^{2}}}{2}a\nabla \phi +u\otimes a:\phi \nabla u+u\otimes u:a\nabla \phi +\\
\label{defk}&&\left. +uf\phi +g:\phi \nabla u+g:u\otimes \nabla \phi  \right]dxds,\\
\xi(t^+)&=&{\lim\sup}_{s\to t^+}\,\xi(s).
\end{eqnarray*}
\item Define $\tilde{e}:I\to\mathbb{R}$, \[{\tilde{e}}(t)=\underset{s\in (c ,t),\,y\in {{\mathbb{R}}^{3}}}{\mathop{\sup }}\,\int_{{{\mathbb{R}}^{3}}}{\frac{|u(x,s){{|}^{2}}}{2}\phi (x-y)dx}.\]
Then $\tilde{e}$ is continuous on $I$.
\item Define $e:I\to\mathbb{R}$, \[e(t)=\underset{s\in (c ,t),\,y\in {{\mathbb{R}}^{3}}}{\mathop{\sup }}\,\left( \int_{{{\mathbb{R}}^{3}}}{\frac{|u(x,s){{|}^{2}}}{2}\phi (x-y)}dx+\int_{c }^{s}{\int_{{{\mathbb{R}}^{3}}}{|\nabla u(x,\tau ){{|}^{2}}\phi (x-y)dxd\tau }} \right).\]
Then $e$ is continuous on $I$. Moreover, if $\lim_{s\to c^+}\|u(s)-u(c)\|_{L^2(B)}= 0$ for every bounded set $B\subset\mathbb{R}^3$ then the function $e$ can be extended to a continuous function on $I\cup\{c\}$ with \[e(c )\coloneqq\underset{y\in {{\mathbb{R}}^{3}}}{\mathop{\sup }}\,\int_{{{\mathbb{R}}^{3}}}{\frac{|u(x,c ){{|}^{2}}}{2}\phi (x-y)dx}.\]
\end{enumerate}
\end{lem}
\begin{proof}
(i) The estimate is proved by applying \eqref{326191} with the test function $\psi_\delta(x,s)=\chi_\delta(s)\phi(x)$, where $\chi_\delta$ is a nonnegative function supported on $(t_0,t+\delta)$ and equal to 1 on $(t_0+\delta,t)$, and then letting $\delta\to 0$. 

(ii) For arbitrary $y\in\mathbb{R}^3$, put $\phi_y(x)=\phi(x-y)$. Let $\xi_y$, $\rho_y$, $k_y$ be defined similarly to $\xi$, $\rho$, $k$, except that $\phi$ is replaced by $\phi_y$. Let $t_0\in I$ and $(t_n)$ be an increasing sequence converging to $t_0$. 
By \autoref{prop131192}, $\liminf\,{\tilde{e}}({{t}_{n}})\ge\liminf\,{{\xi }_{y}}({{t}_{n}})\ge \xi_y(t_0)$.
%
Also note that $\tilde{e}$ is nondecreasing. Thus, $\tilde{e}$ is left-continuous at $t_0$. By Part (i), \begin{equation}\label{112193}\xi_y(s')\le \xi_y(s^+)+k_y(s,s')~~~~\forall s,\,s'\in I,~ s<s'\end{equation}
By Part (i) of \autoref{generalizedsol} and H\"{o}lder inequality, $k_y(s,s')\le A|s'-s|^\theta$ for all $|s'-s|<1$, where $\theta=\theta(m,q_1,q_2)\in(0,1)$ and $A$ is independent of $y$. Applying \eqref{112193} for $s<t_0<s'$ and letting $s\to t_0^-$,
\begin{equation*}\xi_y(s')\le \underset{s\to t_{0}^{-}}{\mathop{\lim \inf }}\,(\xi_y({{s}^{+}})+k_y(s,s'))\le \tilde{e}({{t}_{0}})+k_y({{t}_{0}},s')\ \ \ \ \forall s'>{{t}_{0}}.\end{equation*}
By taking supremum both sides over $s'\in[t_0,t)$
and $y\in\mathbb{R}^3$, we obtain $\tilde{e}(t)\le \tilde{e}({{t}_{0}})+A|t_0-t|^\theta$. Thus, 
$\tilde{e}$ is right-continuous at $t_0$.

(iii) By Part (i),
\[{{\xi }_{y}}(t)+{{\rho }_{y}}(t)\le {{\xi }_{y}}(c^{+})+{{\rho }_{y}}(c)+A(t-c)^\theta~~~\forall\,c<t<d.\]
One can use arguments similar to Part (ii) to show that $e$ is continuous at $c$. Since $u(t)\to u(c)$ in $L^2(B_r)$ for any $R>0$,
\[{{\xi }_{y}}({{c }^{+}})=\int_{{{\mathbb{R}}^{3}}}{\frac{|u(x,c ){{|}^{2}}}{2}{{\phi }_{y}}(x)dx}.\]
Thus, $e(c )\le \underset{s\to {{c }^{+}}}{\mathop{\lim \inf }}\,e(s)\le e(t)\le e(c )+A{{(t-c )}^{\theta }}$ for all $t>c$.
\end{proof}

\subsection{Local energy estimate for suitable weak solutions}

Let $\phi_r:\mathbb{R}^3\to\mathbb{R}$ be a nonnegative smooth function supported on $B_{2r}$, equal to 1 in $B_r$, with derivatives $|\nabla\phi_r|\le Cr^{-1}$ and $|\nabla^2\phi_r|\le Cr^{-2}$. One can take, for example, $\phi_r(x)=\phi_1(x/r)$ where $\phi_1$ is a nonnegative smooth function supported on $B_2$ and equal to 1 on $B_1$. Denote
\begin{equation}\label{funcphi}\phi_{r,y}(x)=\phi_r(x-y)=\phi_1\left(\frac{x-y}{r}\right).\end{equation}
\begin{prop}\label{prop25191}
Let $(u,p)$ be a suitable weak solution to the generalized Navier-Stokes equations on $\mathbb{R}^3\times I$, where $I=(c,c+r^2)$, with $a\in L^m$, $f\in L^{q_1}$, $g\in L^{q_2}$ with $m\ge 5$, $q_1\ge 3/2$, $q_2\ge 2$. For $t\in I$, denote
\[e_r(t)=\underset{s\in (c ,t),y\in {{\mathbb{R}}^{3}}}{\mathop{\sup }}\,\left( \int_{{{\mathbb{R}}^{3}}}{\frac{|u(x,s){{|}^{2}}}{2}{{\phi }_{r,y}}(x)}dx+\int_{c }^{s}{\int_{{{\mathbb{R}}^{3}}}{|\nabla u(x,s' ){{|}^{2}}{{\phi }_{r,y}}(x)dxds' }} \right).\]
Then \[e_r(t)\le e_r(c^+)+\alpha_r({t})e_r(t)^{1/2}+\beta_r({t})e_r(t)+\gamma_r({t})e_r(t)^{3/2}.\]
Moreover, there exists a function $\bar{p}=\bar{p}(y,s)$ such that
\begin{eqnarray*}\underset{y\in {{\mathbb{R}}^{3}}}{\mathop{\sup }}\,\int_{c }^{t}{\int_{{{\mathbb{R}}^{3}}}{|p(x,s)-\bar{p}(y,s){{|}^{\frac{3}{2}}}\phi_{r,y}dxds}}\lesssim \kappa _{1,r}^{\frac{3}{2}}{{r}^{\frac{3}{2}}}{{(\tau{{r}^{3}})}^{1-\frac{3}{2{{q}_{1}}}}}+\kappa _{2,r}^{\frac{3}{2}}{{(\tau{{r}^{3}})}^{1-\frac{3}{2{{q}_{2}}}}}\\
+\kappa _{0,r}^{\frac{3}{2}}\left( {{{\tau}}^{\frac{1}{2}}}{{r}^{-\frac{3}{4}}}+{{{\tau}}^{\frac{1}{8}}} \right){{(\tau{{r}^{3}})}^{\frac{1}{2}-\frac{3}{2m}}}{{e}_{r}}{{(t)}^{\frac{3}{4}}}\\
+\left(\tau{{r}^{-\frac{3}{2}}}+{{{\tau}}^{\frac{1}{4}}} \right){{e}_{r}}{{(t)}^{\frac{3}{2}}},
\end{eqnarray*}
where $\tau=t-c$ and
\[{{\kappa }_{0,r}}=\underset{y\in {{\mathbb{R}}^{3}}}{\mathop{\sup }}\,{{\left\| a \right\|}_{{{L}^{m}}({Q})}},\ \ {{\kappa }_{1,r}}=\underset{y\in {{\mathbb{R}}^{3}}}{\mathop{\sup }}\,{{\left\| f \right\|}_{{{L}^{q_1}}({Q})}},\ \ {{\kappa }_{2,r}}=\underset{y\in {{\mathbb{R}}^{3}}}{\mathop{\sup }}\,{{\left\| g \right\|}_{{{L}^{q_2}}({Q})}},\ \ {Q}={{B}_{2r}}(y)\times(c,t),\]
\[{{\alpha }_{r}}(t)=C{{\kappa }_{1,r}}\left( {\tau^{\frac{1}{3}}}{{r}^{-\frac{1}{2}}}+{\tau^{\frac{1}{12}}} \right){{(\tau{{r}^{3}})}^{\frac{2}{3}-\frac{1}{{{q}_{1}}}}}+C{{\kappa }_{2,r}}\left( 1+{\tau^{\frac{1}{2}}}{{r}^{-1}} \right){{(\tau{{r}^{3}})}^{\frac{1}{2}-\frac{1}{{{q}_{2}}}}},\]
\[{{\beta }_{r}}(t)=C\tau{{r}^{-2}}+C{{\kappa }_{0,r}}\left( {\tau^{\frac{4}{5}}}{{r}^{-\frac{8}{5}}}+1 \right){{(\tau{{r}^{3}})}^{\frac{1}{5}-\frac{1}{m}}},\]
\[{{\gamma }_{r}}(t)=C\left(\tau{{r}^{-\frac{5}{2}}}+{\tau^{\frac{1}{4}}}{{r}^{-1}} \right).\]
\end{prop}
\begin{proof}
We will use the notations $\xi_y$, $\rho_y$, $k_y$ introduced in \autoref{prop21191}. One can assume $c=0$. For convenience, the subscript $r$ in $e_r$, $\phi_{r,y}$, $\kappa_{0,r}$, etc.\ will be dropped during the proof. By Part (i) of \autoref{prop21191}, $e(t)\le e(0^+)+\sup_{y\in\mathbb{R}^3}k_y(0,t)$. It suffices to derive an upper bound of $k_y(0,t)$ that is independent of $y$. Fix $y\in\mathbb{R}^3$. By integration by parts,
\[{{k}_{y}}(0 ,t)=\{1\}+\{2\}+\{3\}+\{4\}+\{5\}+\{6\}+\{7\}+\{8\}+\{9\}\]
where
\[\{1\}=\int_{0 }^{t}{\int_{{{\mathbb{R}}^{3}}}{\frac{|u{{|}^{2}}}{2}\Delta {{\phi }_{y}}dxds}},~~~~~~~\{5\}=\int_{0 }^{t}{\int_{{{\mathbb{R}}^{3}}}{u\otimes a:{{\phi }_{y}}\nabla udxds}},\]
\[\{2\}=\int_{0 }^{t}{\int_{{{\mathbb{R}}^{3}}}{\frac{|u{{|}^{2}}}{2}u\nabla {{\phi }_{y}}dxds}},~~~~~~\{6\}=\int_{0 }^{t}{\int_{{{\mathbb{R}}^{3}}}{u\otimes u:a\nabla {{\phi }_{y}}dxds}},\]
\[\{3\}=\int_{0 }^{t}{\int_{{{\mathbb{R}}^{3}}}{pu\nabla {{\phi }_{y}}dxds}},~~~~~~~~~~~~~~~~~~~~\{7\}=\int_{0 }^{t}{\int_{{{\mathbb{R}}^{3}}}{uf{{\phi }_{y}}dxds}},\]
\[\{4\}=\int_{0 }^{t}{\int_{{{\mathbb{R}}^{3}}}{\frac{|u{{|}^{2}}}{2}a\nabla {{\phi }_{y}}dxds}},~~~~~~~~~~~\{8\}=\int_{0 }^{t}{\int_{{{\mathbb{R}}^{3}}}{g:{{\phi }_{y}}\nabla udxds}},\]
\[\{9\}=\int_{0 }^{t}{\int_{{{\mathbb{R}}^{3}}}{g:u\otimes \nabla {{\phi }_{y}}dxds}}.\]
By standard applications of H\"{o}lder inequality, the Sobolev embedding $H^1\hookrightarrow L^6$, and the fact that $L^2\cap L^6$ is continuously embedded in $L^{10/3}$ and $L^3$, we get
\begin{eqnarray*}
\{1\}&\lesssim&  t{{r}^{-2}}e(t).\\
\{2\}&\le& {{r}^{-1}}\int_{0}^{t}{\int_{{{B}_{2r}}(y)}{\frac{|u{{|}^{3}}}{2}dxds}}\lesssim (t{{r}^{-5/2}}+{{t}^{1/4}}{{r}^{-1}})e{{(t)}^{3/2}}.\\
\{4\},\,\{6\}&\lesssim& \kappa_0\left\| u \right\|_{{{L}^{3}}}^{2}{{r}^{-3/m}}{{t}^{1/3-1/m}}\lesssim \kappa_0\left( {{t}^{2/3}}{{r}^{-1}}+{{t}^{1/6}} \right){{r}^{-3/m}}{{t}^{1/3-1/m}}e(t).\\
\{5\}&\lesssim& ({{r}^{-3/5}}{{t}^{3/10}}+1){{\left\| a \right\|}_{L^{5}}}e(t)\lesssim {{\kappa }_{0}}({{r}^{-3/5}}{{t}^{3/10}}+1){{(t{{r}^{3}})}^{1/5-1/m}}e(t).\\
\{7\}&\le& {{\left\| u \right\|}_{L^{3}}}{{\left\| f \right\|}_{L^{3/2}}}\lesssim {{\kappa }_{1}}({{t}^{1/3}}{{r}^{-1/2}}+{{t}^{1/12}}){{(t{{r}^{3}})}^{2/3-1/{{q}_{1}}}}e{{(t)}^{1/2}}.\\
\{8\}&\le& {{\left\| g \right\|}_{L^{2}}}{{\left\| \nabla u \right\|}_{L^{2}}}\le {{\kappa }_{2}}{{(t{{r}^{3}})}^{1/2-1/{{q}_{2}}}}e{{(t)}^{1/2}}.\\
\{9\}&\lesssim& {{r}^{-1}}{{\left\| g \right\|}_{L^{3/2}}}{{\left\| u \right\|}_{L^{3}}}\lesssim {{\kappa }_{2}}{{r}^{-1}}{{(t{{r}^{3}})}^{2/3-1/{{q}_{2}}}}({{t}^{1/3}}{{r}^{-1/2}}+{{t}^{1/12}})e{{(t)}^{1/2}}.
\end{eqnarray*}
It remains to estimate term \{3\}. Let $R_1$, $R_2$, $R_3$ be the Riesz operators on $\mathbb{R}^3$. Denote $\mathfrak{R}_{ij}=R_i\otimes R_j$. This is a Calder\'{o}n-Zygmund operator with kernel $k(x)=-\nabla^2\Phi$, where $\Phi(x)=-C|x|^{-1}$ is the fundamental solution to the Laplace equation in $\mathbb{R}^3$. Extend $a$, $f$ and $g$ by zero outside of $Q$. The generalized Navier-Stokes equations imply
\[p=\underbrace{\mathfrak{R}(u\otimes u)}_{{{p}^{(1)}}}+\underbrace{\mathfrak{R}(a\otimes u+u\otimes a)}_{{{p}^{(2)}}}+\underbrace{\mathfrak{R}(g)}_{{{p}^{(3)}}}+\underbrace{{{\Delta }^{-1}}\operatorname{div}f}_{{{p}^{(4)}}}.\]
The $L^{3/2}$-norms of the first three terms are estimated thanks to \autoref{presbound2}  as follows.
\begin{equation*}
\label{22191}{{\left\| {{p}^{(1)}}-{{p}^{(1)}}(y,\cdot) \right\|}_{{{L}^{3/2}}(Q)}}\lesssim\underset{z\in {{\mathbb{R}}^{3}}}{\mathop{\sup }}\,{{\left\| u\otimes u \right\|}_{{{L}^{3/2}}(Q)}}\lesssim ({{t}^{2/3}}{{r}^{-1}}+{{t}^{1/6}})e(t).
\end{equation*}
The estimates for $p^{(2)}$ and $p^{(3)}$ are carried out similarly. Note that $\|\nabla p^{(4)}\|_{L^{3/2}_x}=\|f-\mathbb{P}f\|_{L^{3/2}_x}\lesssim {{\left\| f \right\|}_{L^{3/2}_x}}\lesssim r^{3-3/q_1}\|f\|_{L^{q_1}_x}$. By Poincar\'{e} inequality, 
\begin{equation*}\label{22194}{{\left\| {{p}^{(4)}}-{{[{{p}^{(4)}}]}_{{{B}_{2r}}(y)}} \right\|}_{{{L}^{3/2}}(Q)}}\lesssim{{\kappa }_{1}}r{{(t{{r}^{3}})}^{2/3-1/{{q}_{1}}}}.\end{equation*}
Put $\bar{p}(y,s)={{p}^{(1)}}(y,s)+{{p}^{(2)}}(y,s)+{{p}^{(3)}}(y,s)+{{[{{p}^{(4)}}(s)]}_{{{B}_{2r}}(y)}}$. By the estimates of the pressure components above,
\begin{eqnarray}
\nonumber{{\left\| p-\bar{p}(y,\cdot) \right\|}_{{{L}^{3/2}}(Q)}}\lesssim{{\kappa }_{1}}r{{(t{{r}^{3}})}^{\frac{2}{3}-\frac{1}{{{q}_{1}}}}}+{{\kappa }_{2}}{{(t{{r}^{3}})}^{\frac{2}{3}-\frac{1}{{{q}_{2}}}}}+\\
\label{22195}+{{\kappa }_{0}}{{(t{{r}^{3}})}^{\frac{1}{3}-\frac{1}{m}}}({{t}^{\frac{1}{3}}}{{r}^{-\frac{1}{2}}}+{{t}^{\frac{1}{12}}})e{{(t)}^{1/2}}+({{t}^{\frac{2}{3}}}{{r}^{-1}}+{{t}^{\frac{1}{6}}})e(t).
\end{eqnarray}
We estimate \{3\} by the H\"{o}lder inequality $|\{3\}|\lesssim{{\left\| p-\bar{p}(y,\cdot) \right\|}_{L^{3/2}}}{{\left\| u \right\|}_{L^{3}}}{{r}^{-1}}$ and \eqref{22195}.
Combining the estimate of \{1\}, \{2\},\ldots, \{9\}, we get the desired estimate for $e(t)$.
\end{proof}
\begin{cor}\label{prop25192}
Let $M>0$, $r=\frac{\kappa}{M}$ and $\tau=\theta r^2$ where $0<\kappa,\,\theta\le 1$. Let $I=(c,c+\tau)$ for some $c\ge 0$. Suppose $(u,\,p)$ be a mild in solution to (NSE) with $u\in C([c,c+\tau),L^\infty)$. Under conditions \eqref{condensedcond}, \eqref{localc1}, \eqref{localc2}, suppose (NSE)$_\varepsilon$ has a weak solution $(u_\varepsilon,\,p_\varepsilon)$ satisfying \eqref{uecond}. Put $v=u-u_\varepsilon$, $q=p-p_\varepsilon$, and
\[e(t)=\underset{s\in (c,t),\,y\in {{\mathbb{R}}^{3}}}{\mathop{\sup }}\,\left( \int_{{{\mathbb{R}}^{3}}}{\frac{|v(x,s){{|}^{2}}}{2}{{\phi }_{r,y}}(x)}dx+\int_{c }^{s}{\int_{{{\mathbb{R}}^{3}}}{|\nabla v(x,s' ){{|}^{2}}{{\phi }_{r,y}}(x)dxds' }} \right).\]
Then for all $t\in I$,
\begin{equation}e(t)\le e(c^+)+\alpha e{{(t)}^{1/2}}+\beta e(t)+\gamma e{{(t)}^{3/2}}\end{equation}
where
\begin{eqnarray}
\label{328191}\alpha &=&C[(\varepsilon M)^{\nu_1}+(\varepsilon M)^{\sigma_1}]M^{-1/2}\kappa^{1/2},\\ 
\label{328192}\beta &=&C\theta^{1/5},\\
\label{328193}\gamma &=&CM^{1/2}\kappa^{-1/2}.
\end{eqnarray}
Moreover, there exists a function $\bar{q}=\bar{q}(y,s)$ such that \begin{equation*}\underset{y\in {{\mathbb{R}}^{3}}}{\mathop{\sup }}\,\int_{c}^{t}{\int_{{{\mathbb{R}}^{3}}}{|q(x,s)-\bar{q}(y,s){{|}^{3/2}}\phi_{r,y}dxds}}\le \alpha'+\beta'{{e}}{{(t)}^{3/4}}+\gamma'{{e}}{{(t)}^{3/2}}\ \ \forall t\in I
\end{equation*}
where $\alpha'=C[(\varepsilon M)^{\nu_1}+(\varepsilon M)^{\sigma_1}]^{3/2}M^{-2}\kappa^3$, $\beta'=C\kappa^{9/4}M^{-5/4}$, $\gamma'=CM^{-1/2}\kappa^{1/2}\theta^{1/4}$.
\end{cor}
\begin{proof}
By the regularity of $u$ and $u_\varepsilon$, $(v,q)$ is a suitable weak solution to the generalized Navier-Stokes equations
\[\left\{ \begin{matrix}
{{\partial }_{t}}v-\Delta v+\operatorname{div}(u_\varepsilon\otimes v+v\otimes u_\varepsilon+v\,\otimes v)+\nabla q=-f_\varepsilon-\operatorname{div}g_\varepsilon,  \\
\operatorname{div}\,v=0  \\
\end{matrix} \right.{~~~\rm in~~}\mathbb{R}^3\times (c,c+\tau')\]
for all $0<\tau'<\tau$. In \autoref{prop25191}, take $a=u_\varepsilon$ and $m=5$. We have $\kappa_{0,r}\lesssim \theta^{1/5}\kappa$. By \eqref{localc1} and \eqref{localc2},
\begin{eqnarray*}\kappa_{1,r}&\lesssim& {{\varepsilon }^{\nu_1}}r^{\nu_2}{{M}^{\nu_3}}\le (\varepsilon M)^{\nu_1}M^{3-5/q_1},\\
\kappa_{2,r}&\lesssim& {{\varepsilon }^{\sigma_1}}r^{\sigma_2}{{M}^{\sigma_3}}\le (\varepsilon M)^{\sigma_1}M^{2-5/q_2}.\end{eqnarray*} 
Applying these upper bounds for $\kappa_{0,r}$, $\kappa_{1,r}$, $\kappa_{2,r}$ in \autoref{prop25191}, we obtain simple upper bounds $\alpha$, $\beta$, $\gamma$ given by \eqref{328191}-\eqref{328193} for $\alpha_{r}(t)$, $\beta_{r}(t)$, $\gamma_{r}(t)$ respectively.
\end{proof}


\subsection{$\epsilon$-regularity criterion}
Jia and \v Sver\'{a}k proved an $\epsilon$-regularity criterion for the generalized NSE \cite{jiasverak2014}. Here we give a simple generalization of their result in the case the parabolic cylinder $Q_r$ is replaced by $Q_{r,\theta}$, which is needed for our problem.

\begin{lem}\label{prop12281}
Let $f\in L^{q_1}(Q_1)$, $g\in L^{q_2}(Q_1)$, $a\in L^m(Q_1)$ with div\,$a=0$ where $q_1>5/2$, $q_2>5$, $m>5$. Let $(u,p)$ be a suitable weak solution to (GNSE) on $Q_1$. There exists $\varepsilon=\varepsilon(m,q_1,q_2)>0$ such that if \[{{\left\| u \right\|}_{{{L}^{3}}({{Q}_{1}})}}+{{\left\| p \right\|}_{{{L}^{3/2}}({{Q}_{1}})}}+{{\left\| a \right\|}_{{{L}^{m}}({{Q}_{1}})}}+{{\left\| f \right\|}_{{{L}^{{{q}_{1}}}}({{Q}_{1}})}}+{{\left\| g \right\|}_{{{L}^{{{q}_{2}}}}({{Q}_{1}})}}<\varepsilon \] then $u$ is $\alpha$-H\"{o}lder continuous on $Q_{1/2}$ for all
\[0<\alpha<\min \left\{ 2-\frac{5}{{{q}_{1}}},\,1-\frac{5}{{{q}_{2}}},\,1-\frac{5}{m} \right\}.\]
Moreover, $\|u\|_{C^\alpha_{\rm par}(Q_{1/2})}\le C(m,q_1,q_2,\alpha,\varepsilon)$.
\end{lem}
We refer the readers to \cite[Lemma 2.2]{jiasverak2014} for a proof of this lemma.

\begin{prop}[$\epsilon$-regularity]\label{prop117191}
Consider functions $f\in L^{q_1}(Q_{r,\theta})$, $g\in L^{q_2}(Q_{r,\theta})$, $a\in L^m(Q_{r,\theta})$ with div $a=0$ where $q_1>5/2$, $q_2>5$, $m>5$, $0<\theta\le 1$ and $r>0$. Let $(u,p)$ be a suitable weak solution to the generalized NSE on $Q_{r,\theta}(z_0)$. Let $\varepsilon>0$ be the constant found in \autoref{prop12281}. If 
	\begin{eqnarray}
\nonumber		&&\frac{1}{{{r}^{2}}}\int_{{{Q}_{r,\theta}}({{z}_{0}})}{\left( |u{{|}^{3}}+|p{{|}^{3/2}} \right)dz}+r^{m-5}\int_{{{Q}_{r,\theta}}({{z}_{0}})}{|a{{|}^{m}}dz}+\\
\label{117191}		&&+{{r}^{3q_1-5}}\int_{{{Q}_{r,\theta}}({{z}_{0}})}{|f{{|}^{q_1}}dz}+{{r}^{2q_2-5}}\int_{{{Q}_{r,\theta}}({{z}_{0}})}{|g{{|}^{q_2}}dz}<\theta\varepsilon
	\end{eqnarray} then $u$ is $\alpha$-H\"{o}lder continuous on $Q_{r/2,\theta}(z_0)$ for all
	\[0<\alpha<\min \left\{ 2-\frac{5}{{{q}_{1}}},\,1-\frac{5}{{{q}_{2}}},\,1-\frac{5}{m} \right\}.\]
	Moreover, \[{{\left\| u \right\|}_{{{L}^{\infty }}({{Q}_{r/2,\theta}}({{z}_{0}}))}}\le \frac{{{C}_{m,{{q}_{1}},{{q}_{2}}}}}{\sqrt{\theta}r}\,;~~~~[u{{]}_{C_{\text{par}}^{\alpha }({{Q}_{r/2,\theta}}({{z}_{0}}))}}\le \frac{{{C}_{m,{{q}_{1}},{{q}_{2}},\alpha }}}{{\theta{r}^{1+\alpha }}}.\]
\end{prop}
\begin{proof}
Write $z_0=(x_0,t_0)$ and denote $\lambda=\sqrt{\theta}r$. For each $x_1\in\mathbb{R}^3$ such that $|x_1-x_0|\le(1-\sqrt{\theta})r$, put $z_1=(x_1,t_0)$. Note that $Q_{\lambda}(z_1)\subset Q_{r,\theta}(z_0)$. Inequality \eqref{117191} implies
\begin{eqnarray}
\nonumber		&&\frac{1}{{{\lambda}^{2}}}\int_{{{Q}_{\lambda}}({{z}_{1}})}{\left( |u{{|}^{3}}+|p{{|}^{3/2}} \right)dz}+{{\lambda}^{m-5}}\int_{{{Q}_{\lambda}}({{z}_{1}})}{|a{{|}^{m}}dz}+\\
\nonumber		&&+{{\lambda}^{3{{q}_{1}}-5}}\int_{{{Q}_{\lambda}}({{z}_{1}})}{|f{{|}^{{{q}_{1}}}}dz}+{{\lambda}^{2{{q}_{2}}-5}}\int_{{{Q}_{\lambda}}({{z}_{1}})}{|g{{|}^{{{q}_{2}}}}dz}<\varepsilon 
	\end{eqnarray}
By \autoref{prop12281} and the translation and scaling symmetry of the generalized Navier-Stokes equations, $u$ is $\alpha$-H\"{o}lder continuous on $Q_{\rho/2}(z_1)$.
	Moreover, \[{{\left\| u \right\|}_{{{L}^{\infty }}({{Q}_{\lambda/2}}({{z}_{1}}))}}\le \frac{{{C}_{m,{{q}_{1}},{{q}_{2}},\varepsilon }}}{\lambda}\,;~~~~[u{{]}_{C_{\text{par}}^{\alpha }({{Q}_{\lambda/2}}({{z}_{1}}))}}\le \frac{{{C}_{m,{{q}_{1}},{{q}_{2}},\alpha ,\varepsilon }}}{{{\lambda}^{1+\alpha }}}.\]
The proof is then completed by standard covering argument.
\end{proof}

\subsection{Proof of \autoref{mainthm2}}
\noindent The proof uses the following lemma.

\begin{lem}\label{prop25194}
Let $I=(c,c+\tau)$ and $e:I\to\mathbb{R}$ be a continuous nondecreasing function satisfying
\[e(t)\le e(c^+)+\alpha e(t)^{1/2}+\beta e(t)+\gamma e(t)^{3/2}~~~\forall\, t\in I,\]
where $\alpha$, $\beta$, $\gamma>0$. Suppose $\beta<1/2$ and $\alpha\gamma<1/16$. We have the following statements.
\begin{enumerate}[(i)]
\item If $e(c^+)=0$ then $e(c+\tau^-)<16\alpha^2$.
\item If $0<e(c^+)<16\alpha^2$ and $\alpha\gamma<1/64$ then $e(c+\tau^-)<64\alpha^2$.
\item If $16\alpha^2\le e(c^+)\le \frac{1}{256\gamma^2}$ then $e(c+\tau^-)< 4e(c^+)$.
\end{enumerate}
\end{lem}
Here $e(a^-)=\limsup_{t\to a^-}e(t)$. The proof uses similar arguments as in the proof of \autoref{prop28191} and is omitted.

\begin{proof}[Proof of \autoref{mainthm2}]
Put $v = u-u_\varepsilon$ and $q=p-p_\varepsilon$. As long as $u$ remains bounded on an interval $(0,T')$, $(v,q)$ is a suitable weak solution of the generalized Navier-Stokes system
\[\left\{ \begin{matrix}
{{\partial }_{t}}v-\Delta v+\operatorname{div}(u_\varepsilon\otimes v+v\otimes u_\varepsilon+v\,\otimes v)+\nabla q=-f_\varepsilon-\operatorname{div}g_\varepsilon,  \\
\operatorname{div}\,v=0,  \\
\end{matrix} \right.{~~~\rm in~~}\mathbb{R}^3\times (0,T').\]
Put $r=\frac{\kappa}{M}$ and $\tau=\theta r^2$, where $0<\theta,\kappa< 1$ are small numbers to be chosen later. Let $\bar{C}=C_{m=6,q_1,q_2}>2$ be the constant found in \autoref{prop117191}. Let $C_0\in(0,1/4)$ be the absolute constant identified in \autoref{constantc}. 
Let $N>4\bar{C}^2/C_0$ be an integer. We assume that $\theta$ is chosen such that $K=N\frac{TM^2}{\theta\kappa^2}+1$ is an integer greater than N. For $1\le k\le K$, put $t_k=\frac{(k-1)\tau}{N}$. Assume the relation $\varepsilon\le \delta_2M^{-1}\exp(-\mu_2TM^2)$ where $\delta_2,\,\mu_2>0$ are to be chosen. We show by induction on $N\le j\le K$ that
\begin{enumerate}[~~~~~(a)]
\item The mild solution $u$ exists on $(0,t_j)$ and $\|u\|_{L^\infty(\mathbb{R}^3\times(0,t_j))}<\frac{2\bar{C}}{\sqrt{\theta}\kappa}M$,
\item $e^{(j)}(t_{j})\le 4^{j-N}64\alpha^2$, where
\end{enumerate}
\[e^{(k)}(t)=\underset{s\in (t_{k-N+1},t),\,y\in {{\mathbb{R}}^{3}}}{\mathop{\sup }}\,\left( \int_{{{\mathbb{R}}^{3}}}{\frac{|v(x,s){{|}^{2}}}{2}{{\phi }_{r,y}}}dx+\int_{t_{k-N+1} }^{s}{\int_{{{\mathbb{R}}^{3}}}{|\nabla v(x,s' ){{|}^{2}}{{\phi }_{r,y}}dxds' }} \right)\]
and $\phi_{r,y}$ is the bump function defined by \eqref{funcphi}.

If $\theta\kappa^2\le C_0$ then $u$ remains bounded by $2M$ at least until time $\tau$. 
Choose $\mu_2=\frac{2N}{\kappa^2\theta\rho}$ where $\rho=\min\{\nu_1,\sigma_1,\lambda_1\}$. Let $\beta$, $\gamma$ be given by \eqref{328192}, \eqref{328193} and $\alpha=C[(\varepsilon M)^{\nu_1}+(\varepsilon M)^{\sigma_1}+(\varepsilon M)^{\lambda_1}]M^{-1/2}\kappa^{1/2}$ where $C$ is the same constant as in \eqref{328191}. Then
\[\alpha \gamma =C[(\varepsilon M)^{\nu_1}+(\varepsilon M)^{\sigma_1}+(\varepsilon M)^{\lambda_1}]\le C{\delta_2^{\rho}}\exp \left( -\rho\mu_2T{{M}^{2}} \right)\le C{\delta_2^{\rho}}{{4}^{-K}}.\]
For small $\delta_2$ and $\theta$, we have $\alpha\gamma<1/64$ and $\beta<1/2$. By \eqref{localc3},
\[{{e}_{N}}(0)\le C{{\varepsilon }^{{{\lambda }_{1}}}}{{r}^{{{\lambda }_{2}}}}{{M}^{{{\lambda }_{3}}}}=C{{(\varepsilon M)}^{{{\lambda }_{1}}}}{{M}^{-1}}<16\alpha^2.\]
The last inequality requires the smallness of $\delta_2=\delta_2(\rho)$.
By \autoref{prop25194}, $e^{(N)}(t_N)\le 64\alpha^2$. Therefore, (a) and (b) are true for $j=N$. Now suppose that (a) and (b) are true for some $N\le j=k<K$. Because $\|u(t_k)\|_{L^\infty}\le \frac{2\bar{C}}{\sqrt{\theta}\kappa}M$, $u$ remains bounded at least until time
\[{{t}_{k}}+C_{0}{{\left( \frac{2\bar{C}M}{\sqrt{\theta }\kappa } \right)}^{-2}}>{{t}_{k}}+\frac{\theta {{\kappa }^{2}}}{N{{M}^{2}}}={{t}_{k}}+\frac{\tau }{N}={{t}_{k+1}}.\]
By \autoref{prop25192},
\[e^{(k)}(t)\le e^{(k)}(t_{k-N+1}^{+})+\alpha e^{(k)}{{(t)}^{1/2}}+\beta e^{(k)}(t)+\gamma e^{(k)}{{(t)}^{3/2}}\]
for all $t\in(t_{k-N+1},t_{k+1})$.
By Part (iii) of \autoref{prop21191}, the function $e^{(k+1)}$ can extend to a continuous function on $[t_{k-N+2},t_{k+2})$ with
\[e^{(k+1)}(t_{k-N+2}^{+})=e^{(k+1)}({{t}_{k-N+2}})\le e^{(k)}({{t}_{k-N+2}})\le e^{(k)}({{t}_{k}})\le {{4}^{k-5}}64{{\alpha }^{2}}<\frac{1}{256\gamma^2}.\]
The last inequality is true provided the smallness of $\delta_2=\delta_2(\rho)$. By \autoref{prop25194}, Part (ii) and (iii),
\[{{e}^{(k+1)}}({{t}_{k+1}})<\max \left\{ 64\alpha^2,4{{e}^{(k+1)}}({{t}_{k-N+2}}) \right\}\le \max \left\{ 64\alpha ^2,{{4}^{k-4}}64\alpha ^2 \right\}={{4}^{k-4}}64\alpha^2.\]
Thus, (b) is true for $j=k+1$. It remains to show that $u$ remains bounded by $\frac{2\bar{C}}{\sqrt{\theta}\kappa}M$ on the interval $[t_k,t_{k+1})$. 
We will apply the $\epsilon$-regularity criterion to show that $v$ is bounded by $\frac{\bar{C}M}{\sqrt{\theta}\kappa}$ on $(t_k,t_{k+1})$. Once this is proved, the proof is completed because for any $t\in(t_k,t_{k+1})$,
\[{{\left\| u(t) \right\|}_{{{L}^{\infty }}}}\le {{\left\| {{u}_{\varepsilon }}(t) \right\|}_{{{L}^{\infty }}}}+{{\left\| v(t) \right\|}_{{{L}^{\infty }}}}\le M\left( 1+\frac{{\bar{C}}}{\sqrt{\theta }\kappa } \right)\le \frac{2\bar{C}}{\sqrt{\theta }\kappa }M.\]
Let $\varepsilon_0>0$ be the constant found in \autoref{prop117191}. Fix $x_0\in\mathbb{R}^3$. Denote
$Q'_{s}=B_s(x_0)\times(t_{k+1}-\theta s^2,t_{k+1})$ for any $s>0$.
Note that $Q'_{r/2}=B_{r/2}(x_0)\times (t_{k+1}-\frac{\tau}{4},t_{k+1})\supset B_{r/2}(x_0)\times [t_{k},t_{k+1})$. Let $\bar{q}=\bar{q}(x_0,s)$ be the function found in \autoref{prop25192}. By \autoref{prop117191}, it suffices to show that
\begin{eqnarray}\nonumber&&\underbrace{\frac{1}{{{r}^{2}}}\int_{{{Q}'_{r}}}{\left( |v{{|}^{3}}+|q-\bar{q}(x_0,\cdot){{|}^{3/2}} \right)dz}}_{\{1\}}+\underbrace{r\int_{{{Q}'_{r}}}{|{{u}_{\varepsilon }}{{|}^{6}}dz}}_{\{2\}}\\
\label{116191}&&+\underbrace{{{r}^{3{{q}_{1}}-5}}\int_{{{Q}'_{r }}}{|f_\varepsilon{{|}^{{{q}_{1}}}}dz}}_{\{3\}}+\underbrace{{{r}^{2{{q}_{2}}-5}}\int_{{{Q}'_{r }}}{|g_\varepsilon{{|}^{{{q}_{2}}}}dz}}_{\{4\}}<{{\varepsilon }_{0}}\theta\end{eqnarray}
with the convention that if $q_1=\infty$ (or $q_2=\infty$), we replace $q_1$ (respectively, $q_2$) in \eqref{116191} by any constant greater than $5/2$ (respectively, 5). 
By \eqref{localc1} and \eqref{localc2}, 
\[\{2\}\le C\theta \kappa^6,\ \ \{3\}\le C\kappa^{5/2},\ \ \{4\}\le C\kappa^{5}.\]
Parameters $\theta$ and $\kappa$ are chosen such that $\{2\}, \{3\}, \{4\}<\varepsilon_0\theta/4$.
By Sobolev embedding,
\[\frac{1}{{{r}^{2}}}\int_{{{Q}'_{r }}}{|v{{|}^{3}}dz}\lesssim {{\left( \frac{e^{(k+1)}({{t}_{k+1}})}{r} \right)}^{3/2}}\lesssim \left(\frac{{{4}^{k-4}}64{{\alpha }^{2}}}{r}\right)^{3/2}\lesssim \delta_2^{3\rho}.\]
By \autoref{prop25192},
\begin{eqnarray*}
\frac{1}{{{r}^{2}}}\int_{{{Q}'_{r }}}{|q-\bar{q}({{x}_{0}},\cdot){{|}^\frac{3}{2}}dz}\lesssim \delta_2^{3\rho/2} +{{\left( \frac{e^{(k+1)}({{t}_{k+1}})}{r} \right)}^{3/4}}+{{\left( \frac{e^{(k+1)}({{t}_{k+1}})}{r} \right)}^{3/2}}\lesssim {\delta_2 ^{3\rho/2}}.
\end{eqnarray*}
For $\delta_2=\delta_2(\rho)$ sufficiently small, $\{1\}<{{{\varepsilon }_{0}}\theta }/{4}$.
Therefore, (a) is true for $j=k+1$, which completes the proof of (a) and (b) for all $N\le j\le K$. 
\end{proof}
\section{Appendix}
\subsection{Auxiliary lemma}\label{app}
\begin{lem}\label{presbound2}
	Let $T$ be a Calder\'{o}n--Zygmund operator with kernel $m:\mathbb{R}^n\to\mathbb{R}$ satisfying $|m(x)|\le \beta|x|^{-n}$ and $|\nabla m(x)|\le \beta|x|^{-n-1}$. Let $p\in(1,\infty)$, $q\in[1,\infty]$, $r>0$, $x_0\in\mathbb{R}^n$. Let $(\Omega,\mu)$ be a positive measure. Consider a measurable function $f:\mathbb{R}^n\times\Omega\to\mathbb{R}$, $f=f(x,t)=f_t(x)$ such that $\alpha=\sup_{y\in \mathbb{R}^n}\|f\|_{L^q_tL^p_x(B_r(y)\times\Omega)}<\infty$. Put $\gamma=\|T\|_{L^p\to L^p}$. Then
	\[{{\left\| Tf_t-Tf_t({{x}_{0}}) \right\|}_{L^q_t{{L}^{p}_x}({{B}_{r}}({{x}_{0}})\times\Omega)}}\le C_{n,p}\alpha(\gamma+\beta). \]
\end{lem}

\begin{proof}
Let $\phi:\mathbb{R}^n\to\mathbb{R}$ be a smooth cut-off function supported on $B_{4r}(x_0)$ and equal to 1 on $B_{2r}(x_0)$. Decompose $Tf$ as follows:
\[Tf(x)=\int_{{{\mathbb{R}}^{n}}}{m(x-y)f(y)dy}=\underbrace{T(\phi f)}_{\{1\}}+\underbrace{T((1-\phi) f)}_{\{2\}}.\]
By the boundedness of $T$, 
${{\left\| \{1\} \right\|}_{L_{t}^{q}L_{x}^{p}({{\mathbb{R}}^{n}}\times \Omega )}}\le \gamma {{\left\| \phi f \right\|}_{L_{t}^{q}L_{x}^{p}({{\mathbb{R}}^{n}}\times \Omega )}}\le {{C}_{n}}\gamma \alpha$. Denote by $S_k$ the spherical shell with inner radius $\frac{k}{2}r$ and outer radius $\frac{k+1}{2}r$. Then
\[|\{2\}-T{{f}_{t}}({{x}_{0}})|\le \frac{{{C}_{n,p}}\beta }{{{r}^{n/p}}}\sum\limits_{k=4}^{\infty }{\frac{1}{{{k}^{n+1}}}{{\left\| {{f}(t)} \right\|}_{{{L}^{p}}({{S}_{k}})}}}.\]
Because $S_k$ can be covered by $C_nk^{n-1}$ balls of radius $r$ (see \cite{MR2105755}), ${{\left\| f \right\|}_{L_{t}^{q}L_{x}^{p}({{S}_{k}}\times \Omega )}}\le {{C}_{n,p}}{{k}^{n-1}}\alpha $. This estimate completes the proof.
\end{proof}

\subsection{Regularized NSE with $[u\nabla u]_\varepsilon=(u*\eta_\varepsilon)\nabla u$}\label{mol1}
With the initial data $u_0\in L^\infty$, suppose the mollified Navier-Stokes equations 
\begin{equation}\label{mollified}{{\partial _t}u - \Delta u + (u*\eta_\varepsilon) \nabla u + \nabla p = 0},\ \ \text{div}\,u=0\end{equation}
have a mild solution $u_\varepsilon\in L^\infty(\mathbb{R}^3\times(0,T))$ bounded by $M$.
Then (\cite[Thm. 4.1]{koch2009})
\[{{\left\| u_\varepsilon(t)*{{\eta }_{\varepsilon }}-u_\varepsilon(t) \right\|}_{L^{\infty }}}\le \varepsilon {{\left\| \nabla u_\varepsilon(t) \right\|}_{L^{\infty }}}\lesssim \varepsilon \max \left\{ {{M}^{2}},{M}/{\sqrt{t}} \right\}.\]
Recall that \eqref{mollified}
is a special case of (NSE)$_\varepsilon$ where $u_{0\varepsilon}=u_0$, $f_\varepsilon=0$ and $g_\varepsilon=(u_\varepsilon-u_\varepsilon*\eta_\varepsilon)\otimes u_\varepsilon$. We have ${{\left\| {{g}_{\varepsilon }}(t) \right\|}_{L^{\infty }}}\lesssim \min \left\{ {{M}^{2}},\max \left\{ \varepsilon {{M}^{3}},{\varepsilon {{M}^{2}}}/{\sqrt{t}} \right\} \right\}$. Then
\begin{eqnarray}
\label{219211}{{\left\| Gg_{\varepsilon }^{\sigma }(t) \right\|}_{L^{\infty }}}&\lesssim& \int_{0}^{t}{\frac{1}{\sqrt{t-s}}{{\left\| {{g}_{\varepsilon }}(s+\sigma ) \right\|}_{L^{\infty }}}ds}\\
\nonumber&\lesssim& \int_{0}^{{{M}^{-2}}}{\frac{1}{\sqrt{t-s}}\frac{\varepsilon {{M}^{2}}}{\sqrt{s}}ds}+\int_{{{M}^{-2}}}^{t}{\frac{1}{\sqrt{t-s}}\varepsilon {{M}^{3}}ds}\lesssim \varepsilon {{M}^{2}}.
\end{eqnarray}
Thus, condition \eqref{globalc} is satisfied.
By \eqref{28195} and \autoref{prop28191}, $\|u(t_0)-u_\varepsilon(t_0)\|_{L^\infty}\lesssim \varepsilon M^2$ for some $t_0\gtrsim M^{-2}$. The conditions \eqref{condensedcond}-\eqref{localc3} are satisfied with $t_0\gtrsim M^{-1}$ as the starting time and $q_2=\infty$, $\sigma_1=\lambda_1=1$, $\sigma_2=0$, $\lambda_2=3/2$, $\sigma_3=3$, $\lambda_3=2$.

\subsection{Regularized NSE with $[u\nabla u]_\varepsilon=P_\varepsilon(u \nabla u)$}\label{mol2}
With the initial data $u_0\in L^\infty$, suppose the regularized Navier--Stokes equations
\begin{equation}\label{regularized}{{\partial _t}u - \Delta u + P_{\varepsilon}(u\nabla u) + \nabla p = 0},\ \ \text{div}\,u=0\end{equation}
have a mild solution $u_\varepsilon\in L^\infty(\mathbb{R}^3\times(0,T))$ bounded by $M$. Recall that \eqref{regularized} is a special case of (NSE)$_\varepsilon$ where $u_{0\varepsilon}=u_0$, $f_\varepsilon=0$ and $g_\varepsilon=(Id-P_\varepsilon)(u_\varepsilon\otimes u_\varepsilon)$. One can rewrite \eqref{regularized} as $u_\varepsilon=e^{t\Delta} u_0-FP_\varepsilon(u_\varepsilon\nabla u_\varepsilon)$. Standard $L^1_x$ estimates on the kernel of the operator $\nabla FP_\varepsilon$ (similar  to that of $\nabla F$) yield
\[{{\left\| \nabla {{u}_{\varepsilon }}(t) \right\|}_{{{L}^{\infty }}}}\lesssim \frac{{{\left\| {{u}_{0}} \right\|}_{{{L}^{\infty }}}}}{\sqrt{t}}+\int_{0}^{t}{\frac{\left\| {{u}_{\varepsilon }}(s) \right\|_{L^\infty}\left\| \nabla {{u}_{\varepsilon }}(s) \right\|_{L^\infty}}{\sqrt{t-s}}ds}.\]
This leads to $\|\nabla u_\varepsilon(t)\|_{L^\infty}\lesssim M/\sqrt{t}$. Consequently, \[\|g_\varepsilon(t)\|_{L^\infty}\lesssim \varepsilon\|\nabla(u_\varepsilon\otimes u_\varepsilon)\|_{L^\infty}\lesssim \varepsilon M^2/\sqrt{t}.\]
Using the estimate \eqref{219211}, one gets ${{\left\| Gg_{\varepsilon }^{\sigma }(t) \right\|}_{{{L}^{\infty }}}}\lesssim \varepsilon {{M}^{2}}$. Thus, condition \eqref{globalc} is satisfied. The conditions \eqref{condensedcond}-\eqref{localc3} are also satisfied with $t_0\gtrsim M^{-2}$ as the starting time and with the same values for $\sigma_1$, $\sigma_2$, \ldots as above.
\section*{Acknowledgment}
The author would like to thank Prof. Vladim\'{i}r \v{S}ver\'{a}k for his valuable comments.
\section*{Conflict of interest}
 The author declares that he has no conflict of interest.
\bibliographystyle{spmpsci}
\bibliography{References}
\end{document}